\documentclass[a4paper]{amsart}

\usepackage{amssymb,amsmath,amsthm} 
\usepackage{color}

\newtheorem{theorem}{Theorem}[section]
\newtheorem{lemma}[theorem]{Lemma}
\newtheorem{corollary}[theorem]{Corollary}
\newtheorem{proposition}[theorem]{Proposition}

\theoremstyle{definition}

\theoremstyle{remark}
\newtheorem{remark}[theorem]{Remark}

\numberwithin{equation}{section}

\newcommand*{\eps}{\varepsilon} 
\DeclareMathOperator*{\sgn}{sgn}
\def\ve{\varepsilon}

\def\Ex{{\mathbb E}}
\def\Pr{{\mathbb P}}
\def\R{{\mathbb R}}
\def\er{{\mathbb R}}
\def\ind{{\mathbf 1}}

\def\supi{\sup_{1\le i \le n}}

\begin{document}

\title[Comparison of weak and strong moments for indep. coordinates]{Comparison of weak and strong moments \\ for vectors with independent coordinates}
\author{Rafa{\l} Lata{\l}a and Marta Strzelecka}
\date{7.12.2016}
\address{Institute of Mathematics, University of Warsaw, Banacha 2, 02--097 Warsaw, Poland.}
\email{rlatala@mimuw.edu.pl, martast@mimuw.edu.pl}
\thanks{The research of RL was supported by the National Science Centre, Poland grant 2015/18/A/ST1/00553 and
of MS by  the National Science Centre, Poland grant 2015/19/N/ST1/02661}



\begin{abstract}
We show that for $p\ge 1$, the $p$-th moment of suprema of linear combinations of independent
centered random variables are comparable with the sum of the first moment and the weak $p$-th moment provided that $2q$-th and $q$-th integral moments of these variables are comparable for all $ q \ge 2$. 
The latest condition turns out to be necessary in the i.i.d. case.   
\end{abstract}

\maketitle

\section{Introduction and Main Results}

In many problems arising in probability theory and its applications one needs to study
variables of the form $S=\sup_{t\in T}|\sum_{i=1}^nt_iX_i|$, where $X_1,\ldots,X_n$ are independent
random variables  and $T$ is a non-empty subset of $\R^n$. In particular it is of interest to estimate tails of $S$. Such estimates are
strictly related to bounds for $L_p$-norms of $S$ (i.e. $\|S\|_p:=(\Ex|S|^p)^{1/p}$) 
for $p\ge 1$ (see Corollary \ref{cor:weakstrongtail} and its proof in Section \ref{sec:wst} below).
There is a trivial lower estimate:
\begin{equation}
\label{eq:lowtrivial}
\biggl(\Ex\sup_{t\in T}\Bigl|\sum_{i=1}^nt_iX_i\Bigr|^p\biggr)^{1/p}
\ge \max\biggl\{
\Ex\sup_{t\in T}\Bigl|\sum_{i=1}^nt_iX_i\Bigr|,
\sup_{t\in T}\biggl(\Ex\Bigl|\sum_{i=1}^nt_iX_i\Bigr|^p\biggr)^{1/p}\biggr\}.
\end{equation}
It turns out that in some situations this obvious lower bound may be reversed, i.e.
there exist numerical constants $C_1$ and $C_2$ such that 
\begin{equation}
\label{eq:weakstrongmom}
\biggl(\Ex\sup_{t\in T}\Bigl|\sum_{i=1}^nt_iX_i\Bigr|^p\biggr)^{1/p}
\le
C_1\Ex\sup_{t\in T}\Bigl|\sum_{i=1}^nt_iX_i\Bigr|
+C_2\sup_{t\in T}\biggl(\Ex\Bigl|\sum_{i=1}^nt_iX_i\Bigr|^p\biggr)^{1/p}.
\end{equation}
This is for example the case (with $C_1=1$), when $X_i$ are normally distributed. This
is an easy consequence of the Gaussian concentration (cf.\ Chapter 3 of \cite{LeTal}).
Dilworth and Montgomery-Smith \cite{DMS} established  the inequality \eqref{eq:weakstrongmom} for $X_i$ being symmetric
Bernoulli random variables. This result was generalized in \cite{La_logconctails} to symmetric variables
with logarithmically concave tails and in \cite[Theorem 2.3]{LaTk} to symmetric random variables such that
$\|X_i\|_{q} \le C\frac{p}{q} \alpha \|X_i\|_p$ for all $q\ge p\ge 2$.

The main result of this paper is the following.

\begin{theorem}
\label{thm:p2p_to_paour} 
Let $X_1,\ldots,X_n$ be independent mean zero random variables with finite moments such that 
\begin{equation}
\label{eq:main_thm_assumpt}
\|X_i\|_{2p} \le \alpha \|X_i\|_p \qquad \mbox{for every $p\ge 2$ and $i=1,\ldots,n$},
\end{equation}
where $\alpha$ is a finite positive constant.
Then for every $p\ge 1$ and every non-empty set $T\subset \R^n$ we have
\begin{equation} 
\label{eq:p2pweakstrong}
\biggl(\Ex  \sup_{t\in T} \Bigl|\sum_{i=1}^n t_iX_i \Bigr|^p \biggr)^{1/p} 
\le C(\alpha) \Biggl[  \Ex  \sup_{t\in T} \Bigl|\sum_{i=1}^n t_iX_i\Bigr|  
+  \sup_{t\in T} \biggl(\Ex   \Bigl|\sum_{i=1}^n t_i X_i \Bigr|^p \biggr)^{1/p} \Biggr] ,
\end{equation}
where $C(\alpha)$ is a constant which depends only on $\alpha$.
\end{theorem}

It turns out that Theorem \ref{thm:p2p_to_paour} may be reversed  in the i.i.d. case.
	
\begin{theorem}
\label{thm:paour_to_p2p}
Let $X_1, X_2,\ldots $ be i.i.d. random variables. Assume that there exists a constant $L$ such that
for every $p\ge 1$, every $n$ and every  non-empty set $T\subset \R^n$ we have
\begin{equation}\label{thm_assump:paour_to_p2p}
\biggl(\Ex  \sup_{t\in T} \Bigl|\sum_{i=1}^n t_iX_i \Bigr|^p \biggr)^{1/p} 
\le L \Biggl[  \Ex  \sup_{t\in T} \Bigl|\sum_{i=1}^n t_iX_i\Bigr|  
+  \sup_{t\in T} \biggl(\Ex   \Bigl|\sum_{i=1}^n t_i X_i \Bigr|^p \biggr)^{1/p} \Biggr] .
\end{equation}
Then
\begin{equation}
\label{eq:comp_p2p}
\|X_1\|_{2p} \le \alpha(L) \|X_1\|_p \qquad \mbox{for } p\ge 2,
\end{equation}
where $\alpha(L)$ is a constant which depends only on $L\ge 1$.
\end{theorem}
	
It will be clear from the proof of Theorem \ref{thm:paour_to_p2p} that it suffices to assume 
\eqref{thm_assump:paour_to_p2p} for $T=\{\pm e_j : j\in\{ 1,\ldots, n\}\}$ only,
where $\{e_1, \ldots, e_n  \}$ is the canonical basis of $\R^n$.

The comparison of weak and strong moments \eqref{eq:p2pweakstrong} yields also a deviation inequality 
for $\sup_{t\in T}|\sum_{i=1}^nt_iX_i|$.

\begin{corollary}
\label{cor:weakstrongtail}
Assume $X_1, X_2, \ldots$ satisfy the assumptions of Theorem \ref{thm:p2p_to_paour}. Then for any $u\ge 0$
 and any non-empty set $T$ in $\er^n$,
\begin{multline}
\label{eq:weakstrongtail}
\Pr\biggl(\sup_{t\in T}\Bigl|\sum_{i=1}^nt_iX_i\Bigr|\ge
C_1(\alpha)\biggl[u+\Ex\sup_{t\in T}\Bigl|\sum_{i=1}^nt_iX_i\Bigr| \biggr]\biggr)
\\
\le 
C_2(\alpha)\sup_{t\in T}\Pr\biggl(\Bigl|\sum_{i=1}^nt_iX_i\Bigr|\ge u\biggr),
\end{multline}
where constants $C_1(\alpha)$ and $C_2(\alpha)$ depend only on the constant $\alpha$ in
\eqref{eq:main_thm_assumpt}.
\end{corollary}

Another consequence of the main theorem is the following Khintchine-Kahane type inequality.

\begin{corollary}
\label{cor:KahKh}
Assume $X_i$, $1\le i\le n$ satisfy the assumptions of Theorem \ref{thm:p2p_to_paour}. Then for any 
$p\ge q\ge 2$ and any non-empty set $T$ in $\er^n$ we have,
\[
\biggl(\Ex  \sup_{t\in T} \Bigl|\sum_{i=1}^n t_iX_i \Bigr|^p \biggr)^{1/p}
\le C_3(\alpha)\biggl(\frac{p}{q}\biggr)^{\max\{ 1/2,\log_2\alpha\}}
\biggl(\Ex  \sup_{t\in T} \Bigl|\sum_{i=1}^n t_iX_i \Bigr|^q \biggr)^{1/q}
\]
where a constant $C_3(\alpha)$ depend only on the constant $\alpha$ in
\eqref{eq:main_thm_assumpt}.
\end{corollary}

\medskip	

We postpone  proofs of the above results and first  present a number of remarks and open questions.

\begin{remark}
Exponent $\max\{1/2,\log_2\alpha\}$ in Corollary \ref{cor:KahKh} is optimal. 
\end{remark}

Indeed, since
$\|g\|_p\sim \sqrt{p/e}$ as $p\to\infty$ one cannot go below $1/2$ by the central limit theorem.

To see that $\log_2\alpha$ term cannot be improved it is enough to consider $\alpha>\sqrt{2}$. 
Let $r=1/\log_2\alpha\in (0,2)$ and let $X$ be a symmetric random variable  given by 
$\Pr (|X|\ge t)=e^{-t^r}$ (with $2>r>0$), i.e. $X=|\mathcal{E}|^{1/r}\sgn \mathcal{E}$, where 
$\mathcal{E}$ has the symmetric exponential distribution.  By Stirling's formula 
$\Gamma(x+1) = (\frac{x}{e})^x \sqrt{2\pi x} e^{f(x)}$ with $f(x)\in(0,1/12)$ for $x\ge 1$, so
for $p\geq 2$,
\begin{equation*}
\frac{\|X\|_{2p}}{\|X\|_p}=\frac{\Gamma \bigl(\frac{2p}r+1\bigr)^{1/(2p)}}{\Gamma\bigl(\frac pr +1\bigr)^{1/p}} 
\le  2^{1/r} \biggl(\frac{r}{\pi p}\biggr)^{1/(4p)} e^{1/(24p)} \le 2^{1/r}=\alpha. 
\end{equation*}

Moreover, $\|X\|_p \sim (\frac p{er})^{1/r}$ for $p\to\infty$, so the assertion of Corollary \ref{cor:KahKh} cannot hold with any exponent better than $\log_2\alpha$.

\begin{remark}
If the variables $X_i$ are symmetric then the
term $\Ex  \sup_{t\in T} \left|\sum_{i=1}^n t_iX_i\right|$ in \eqref{eq:p2pweakstrong} may be replaced by
$\Ex  \sup_{t\in T} \sum_{i=1}^n t_iX_i$. 
\end{remark}

\begin{proof}
Let $s$ be any point in $T$. Then $T\subset T-T+s$, so by the triangle inequality
\[
\biggl(\Ex  \sup_{t\in T} \Bigl|\sum_{i=1}^n  t_iX_i  \Bigr|^p \biggr)^{1/p}  
\le \biggl(\Ex  \sup_{t\in T-T} \Bigl|\sum_{i=1}^n t_iX_i \Bigr|^p \biggr)^{1/p} +  
\biggr(\Ex   \Bigl|\sum_{i=1}^n s_i X_i \Bigr|^p \biggr)^{1/p}.
\]
Estimate \eqref{eq:p2pweakstrong} applied to the set $T-T$ yields
\begin{align*}
\biggl(\Ex  \sup_{t\in T-T} \Bigl|\sum_{i=1}^n t_iX_i \Bigr|^p \biggr)^{1/p}
&\le
C(\alpha)\Biggl[
\Ex  \sup_{t\in T-T} \Bigl|\sum_{i=1}^n t_iX_i\Bigr|
+ \sup_{t\in T-T}\biggl(\Ex   \Bigl|\sum_{i=1}^n t_i X_i \Bigr|^p \biggr)^{1/p}
\Biggr].
\end{align*}

The set $T-T$ is symmetric, so
\[
\Ex  \sup_{t\in T-T} \Bigl|\sum_{i=1}^n t_iX_i\Bigr|
=\Ex  \sup_{t\in T-T} \sum_{i=1}^n t_iX_i\le 2 \Ex  \sup_{t\in T}\sum_{i=1}^n t_iX_i,
\]
where the last estimate follows, since $(X_i)_{i=1}^n$ and $(-X_i)_{i=1}^n$ are equally distributed.
Moreover,
\[
\sup_{t\in T-T}\biggl(\Ex   \Bigl|\sum_{i=1}^n t_i X_i \Bigr|^p \biggr)^{1/p}
\le 2\sup_{t\in T}\biggl(\Ex   \Bigl|\sum_{i=1}^n t_i X_i \Bigr|^p \biggr)^{1/p},
\]	
 what finishes the proof of the remark.
\end{proof}

\begin{remark}
If the variables $X_i$ are not centered then \eqref{eq:p2pweakstrong} holds provided that  the assumption  \eqref{eq:main_thm_assumpt} is replaced by 
\[
\|X_i-\Ex X_i\|_{2p} \le \alpha \|X_i-\Ex X_i\|_p \qquad \mbox{for $p\ge 2$ and $i=1,\ldots,n$}.
\]
\end{remark}

\begin{proof}
We have
\[
\biggl(\Ex  \sup_{t\in T} \Bigl|\sum_{i=1}^n t_iX_i \Bigr|^p \biggr)^{1/p}
\le \biggl(\Ex  \sup_{t\in T} \Bigl|\sum_{i=1}^n t_i(X_i-\Ex X_i) \Bigr|^p \biggr)^{1/p}
+\sup_{t\in T}\Bigl|\sum_{i=1}^n t_i \Ex X_i \Bigr|.
\]
Theorem \ref{thm:p2p_to_paour} applied to centered variables  $X_i-\Ex X_i$, $i=1,\ldots,n$, yields
\begin{multline*}
\biggl(\Ex  \sup_{t\in T} \Bigl|\sum_{i=1}^n t_i(X_i-\Ex X_i) \Bigr|^p \biggr)^{1/p} 
\\
\le C(\alpha) \Biggl[  \Ex  \sup_{t\in T} \Bigl|\sum_{i=1}^n t_i(X_i-\Ex X_i)\Bigr|  
+  \sup_{t\in T} \biggl(\Ex   \Bigl|\sum_{i=1}^n t_i (X_i-\Ex X_i) \Bigr|^p \biggr)^{1/p} \Biggr].
\end{multline*}
To conclude it is enough to observe that
\[
\Ex  \sup_{t\in T} \Bigl|\sum_{i=1}^n t_i(X_i-\Ex X_i)\Bigr|
\le \Ex  \sup_{t\in T} \Bigl|\sum_{i=1}^n t_iX_i\Bigr|
+\sup_{t\in T}\Bigl|\sum_{i=1}^n t_i \Ex X_i \Bigr|,
\]
\[
\sup_{t\in T} \biggl(\Ex   \Bigl|\sum_{i=1}^n t_i (X_i-\Ex X_i) \Bigr|^p \biggr)^{1/p}
\le \sup_{t\in T} \biggl(\Ex   \Bigl|\sum_{i=1}^n t_i X_i \Bigr|^p \biggr)^{1/p}
+\sup_{t\in T}\Bigl|\sum_{i=1}^n t_i \Ex X_i \Bigr|,
\]
and
\[
\sup_{t\in T}\Bigl|\sum_{i=1}^n t_i \Ex X_i \Bigr|\le
\Ex \sup_{t\in T} \Bigl|\sum_{i=1}^n t_i X_i  \Bigr|.
\]
\end{proof}

{\bf Open questions.} For Gaussian random vectors \eqref{eq:weakstrongmom} holds with $C_1=1$.
This is also the case for $X_i$ symmetric, independent with log-concave distributions  \cite[Remark 3.16 and Corollary 2.19]{LaW}.
However, we do not know the general conditions for  the distributions of $X_i$ which are sufficient
for \eqref{eq:weakstrongmom} to hold with $C_1=1$.

It is of interest to study the comparison of weak and strong moments for random vectors 
$X=(X_1,\ldots,X_n)$ with dependent coordinates. 
{A natural and important class to investigate in this context are vectors
with  log-concave distributions (cf.\ \cite{BGVV} for an up to date survey of properties
of such vectors). 
Paouris \cite{Pa} showed that \eqref{eq:weakstrongmom}
holds for log-concave vectors and sets $T$ being balls in Euclidean spaces (see also \cite{MR3150710}). This was generalized in \cite{LaSt} to  balls in $L_r$-spaces with $1\le r<\infty$. Unfortunately there are very few classes of log-concave vectors such that \eqref{eq:weakstrongmom} is known to be satisfied for all sets $T$ -- this
includes vectors uniformly distributed on $l_r^n$-balls (with $1\le r \le \infty$).  \cite[Remark 3.16 and Theorem 5.27]{LaW},
or more generally vectors with densities of the form $\exp(-\varphi(\|x\|_r))$, where 
$\varphi\colon [0,\infty)\to (-\infty,\infty]$
is  non-decreasing and convex, and $1\le r\le \infty$ \cite[Proposition 6.5]{La_sud}.   

\medskip

The organization of this paper is as follows. In Section 2 we prove Theorem \ref{thm:p2p_to_paour}
for unconditional sets $T$ only. Using this result we generalize it to the case of an arbitrary $T$ in Section 3. 
In Section 4 we prove Corollaries \ref{cor:weakstrongtail} and \ref{cor:KahKh}. 
Finally, in Section 5 we present the proof of Theorem \ref{thm:paour_to_p2p}.

\medskip

Throughout this paper by a letter $C$ we denote universal constants and by $C(\alpha)$  constants depending
only on the parameter $\alpha$. The values of the constants $C$, $C(\alpha)$ may differ at each occurrence. 
We will also frequently work with a Bernoulli sequence $\ve_i$ of i.i.d.\ symmetric random variables taking values
$\pm 1$. We assume that variables $\ve_i$ are independent of other random variables.

\section{The case of unconditional sets }

In this section we show that Theorem \ref{thm:p2p_to_paour} holds under additional assumptions that 
the set $T$ is unconditional and  the variables $X_i$ are symmetric. 
Recall that a set $T$ in $\er^n$ is called  \emph{unconditional} if it is symmetric with 
respect to the coordinate axes, i.e. $(\eta_it_i)_{i=1}^n\in T$ for
any $t=(t_i)_{i=1}^n\in T$ and any choice of signs $\eta_1,\ldots,\eta_n\in\{-1,1\}$.

\begin{proposition}
\label{thm:Y_to_X}
Let $r\in (0,1)$ and $L\ge 1$. Assume that variables $Y_1,\ldots,Y_n$ are independent and symmetric and
\begin{equation}
\label{ass:Y_to_X}
\biggl(\Ex  \sup_{t\in T} \Bigl|\sum_{i=1}^n t_iY_i \Bigr|^p \biggr)^{1/p} 
\le 
L \Biggl[  \Ex  \sup_{t\in T} \sum_{i=1}^n t_iY_i  
+  \sup_{t\in T} \biggl(\Ex   \Bigl|\sum_{i=1}^n t_i Y_i \Bigr|^p \biggr)^{1/p} \Biggr] 
\end{equation} 
for all $p\ge 1$ and all unconditional  sets $T$. 
Then variables $X_i:= |Y_i|^{1/r}\sgn Y_i$ satisfy
\begin{equation} 
\label{aim}
\biggl(\Ex  \sup_{t\in T} \Bigl|\sum_{i=1}^n t_iX_i \Bigr|^p \biggr)^{1/p} 
\le 
(2L)^{1/r} \Biggl[  \Ex  \sup_{t\in T} \sum_{i=1}^n t_iX_i  
+  \sup_{t\in T} \biggl(\Ex   \Bigl|\sum_{i=1}^n t_i X_i \Bigr|^p \biggr)^{1/p} \Biggr] 
\end{equation}
for all $p\ge 1$ and all unconditional  sets $T\subset \R^n$.
\end{proposition}

\begin{proof}
Definition of $X_i$ and unconditionality of $T$  yield
\[
\sup_{t\in T} \Bigl|\sum_{i=1}^n  t_iX_i \Bigr|  
=\sup_{t\in T} \Bigl|\sum_{i=1}^n  t_i|Y_i|^{1/r}\sgn Y_i \Bigr| 
=\sup_{t\in T} \sum_{i=1}^n  |t_i||Y_i|^{1/r}. 
\]
Let $s=(1-r)^{-1}$ and let $B_s^n$ denote the unit ball of $\ell_s^n$. Then $1/s+r=1$ and by H\"older's duality
we have
\[
\sup_{t\in T} \Bigl|\sum_{i=1}^n  |t_i||Y_i|^{1/r}\Bigr|^r
=\sup_{t\in T}\sup_{u\in B_s^n} \sum_{i=1}^n  u_i|t_i|^r Y_i
=\sup_{t\in T_r}\sum_{i=1}^n  t_i Y_i, 
\]
where
\[
T_{r}:=\{(u_i|t_i|^r)_{i=1}^n\colon\ t\in T, u\in B_s^n\}
\]
is unconditional in $\er^n$. Therefore \eqref{ass:Y_to_X} applied with $p/r$ and $T_r$ instead of $p$ and $T$ yields
\[
\Ex  \sup_{t\in T} \Bigl|\sum_{i=1}^n  t_iX_i \Bigr|^p 
\le
L^{p/r}\Biggl[
\Ex \sup_{t\in T_r}t_i Y_i
+\sup_{t\in T_r}\biggl(\Ex\Bigl|\sum_{i=1}^n  t_i Y_i\Bigr|^{p/r}\biggr)^{r/p}
\Biggr]^{p/r}.
\]
We have
\[
\Ex \sup_{t\in T_r}\sum_{i=1}^n  t_i Y_i=\Ex  \sup_{t\in T} \Bigl|\sum_{i=1}^n  t_iX_i \Bigr|^{r}
\le 
\biggl(\Ex  \sup_{t\in T} \Bigl|\sum_{i=1}^n  t_iX_i \Bigr|\biggr)^{r}.
\]
Moreover,
\begin{align*}
\sup_{t\in T_r}\biggl(\Ex\Bigl|\sum_{i=1}^n  t_i Y_i\Bigr|^{p/r}\biggr)^{r/p}
&\le
\sup_{t\in T}\biggl(\Ex\sup_{u\in B_s^n}\Bigl|\sum_{i=1}^nu_i|t_i|^r Y_i\Bigr|^{p/r}\biggr)^{r/p}
\\
&=
\sup_{t\in T}\biggl(\Ex\Bigl|\sum_{i=1}^n |t_i| |X_i|\Bigr|^{p}\biggr)^{r/p}
=\sup_{t\in T}\biggl(\Ex\Bigl|\sum_{i=1}^n t_i |X_i|\Bigr|^{p}\biggr)^{r/p}.
\end{align*}
Estimates above together with  the inequality $(a+b)^{1/r}\le 2^{1/r-1}(a^{1/r}+b^{1/r})$ yield
\[
\biggl(\Ex  \sup_{t\in T} \Bigl|\sum_{i=1}^n  t_iX_i \Bigr|^p \biggr)^{1/p}
\le 
\frac{1}{2}(2L)^{1/r}\Biggl[
\Ex  \sup_{t\in T} \Bigl|\sum_{i=1}^n  t_iX_i \Bigr|
+\sup_{t\in T}\biggl(\Ex\Bigl|\sum_{i=1}^n t_i |X_i|\Bigr|^{p}\biggr)^{1/p}
\Biggr].
\]		
			
Hence, in order to prove \eqref{aim} it suffices to show that
\begin{equation}
\label{aim2}
\sup_{t\in T} \biggl(\Ex   \Bigl|\sum_{i=1}^n t_i|X_i| \Bigr|^{p} \biggr)^{1/p} 
\le  
\Ex \sup_{t\in T} \sum_{i=1}^n t_iX_i  +  2\sup_{t\in T} \biggl(\Ex\Bigl|\sum_{i=1}^n t_iX_i \Bigr|^p \biggr)^{1/p} .
\end{equation}

Let $(X_1', \ldots , X_n')$ be an independent copy of $(X_1,\ldots, X_n)$.	
By the triangle inequality for the $p$-th integral norm and Jensen's inequality we get
\begin{align}
\notag
\sup_{t\in T} \biggl(\Ex  \Bigl|\sum_{i=1}^n t_i |X_i| \Bigr|^p \biggr)^{1/p} &
\le 
\sup_{t\in T} \biggl(\Ex   \Bigl|\sum_{i=1}^n t_i  \bigl(|X_i| - \Ex|X_i'| \bigr) \Bigr|^p \biggr)^{1/p} 
+ \sup_{t\in T}\Bigl| \Ex   \sum_{i=1}^n t_i |X_i|\Bigr|
\\ &
\notag
\le  
\sup_{t\in T} \biggl(\Ex   \Bigl|\sum_{i=1}^n t_i  \bigl(|X_i| - |X_i'|\bigr) \Bigr|^p \biggr)^{1/p} +  
\Ex \sup_{t\in T}  \Bigl|\sum_{i=1}^n t_i |X_i|\Bigr|
\\ &
\label{interm_step2}
=
\sup_{t\in T} \biggl(\Ex   \Bigl|\sum_{i=1}^n t_i  \bigl(|X_i| - |X_i'|\bigr) \Bigr|^p \biggr)^{1/p} +  
\Ex \sup_{t\in T}  \sum_{i=1}^n t_i X_i,
\end{align}
where the equation follows by the unconditionality of $T$.
	
Let $(\eps_i)_{i=1}^n$ be the Bernoulli sequence, independent of all $X_i$ and $X_i'$. Then
the sequence $(|X_i|-|X_i'|)_{i=1}^n$ has the same distribution as $(\ve_i(|X_i|-|X_i'|))_{i=1}^n$ and
for every $t\in \R^n$,
\begin{align}
\notag
\biggl(\Ex   \Bigl|\sum_{i=1}^n t_i  (|X_i| - |X_i'|) \Bigr|^p \biggr)^{1/p} 
&= 
\biggl(\Ex   \Bigl|\sum_{i=1}^n t_i  \eps_i (|X_i| - |X_i'|) \Bigr|^p \biggr)^{1/p}
\\ 
\notag
& \le 
\biggl(\Ex   \Bigl|\sum_{i=1}^n t_i  \eps_i |X_i| \Bigr|^p \biggr)^{1/p}+
\biggl(\Ex   \Bigl|\sum_{i=1}^n t_i  \eps_i |X_i'| \Bigr|^p \biggr)^{1/p}
\\
\label{interm_step3}
&= 
2 \biggl(\Ex \Bigl|\sum_{i=1}^n t_i  X_i \Bigr|^p \biggr)^{1/p}.
\end{align}
	
Putting \eqref{interm_step2} and \eqref{interm_step3} together we get \eqref{aim2}, what completes 
the proof of \eqref{aim}.
\end{proof}

\begin{corollary}
\label{cor:p2p_to_paour}
Let $X_1,\ldots,X_n$ be independent symmetric random variables with finite moments such that 
\begin{equation}
\label{comp_p2p}
\|X_i\|_{2p} \le \alpha \|X_i\|_p \qquad \mbox{for $p\ge 2$ and $i=1,\ldots,n$},
\end{equation}
where $\alpha$ is a finite positive constant.
Then for every $p\ge 1$ and every unconditional set $T\subset \R^n$ we have
\begin{equation}
\label{eq:weakstronguncond}
\biggl(\Ex  \sup_{t\in T} \Bigl|\sum_{i=1}^n t_iX_i \Bigr|^p \biggr)^{1/p} 
\le C(\alpha) \Biggl[  \Ex  \sup_{t\in T} \sum_{i=1}^n t_iX_i  
+  \sup_{t\in T} \biggl(\Ex   \Bigl|\sum_{i=1}^n t_i X_i \Bigr|^p \biggr)^{1/p} \Biggr] ,
\end{equation}
where $C(\alpha)$ is a constant, which depends only on $\alpha$.
\end{corollary}

\begin{proof}
Let us first note, that the assumption \eqref{comp_p2p} applied $k$ times yields that 
\begin{equation*}
\| X_i\|_{2^kp} \le \alpha^k \|X_i\|_p \qquad \mbox{for } p\ge 2.
\end{equation*}
Therefore 
\begin{equation*}
\| X_i\|_{q} 
\le \alpha^{\lceil \log_2(\frac{q}p) \rceil} \|X_i\|_p 
\le \alpha \biggl(\frac qp  \biggr)^{\log_2\alpha} \|X_i\|_p 
\qquad \mbox{for } q\ge p\ge 2.
\end{equation*}

Let $Y_i:=|X_i|^{1/\log_2\alpha}\sgn X_i$. Then $X_i = |Y_i|^{1/r}\sgn Y_i$ with $r:=\frac{1}{\log_2\alpha}$ and 
\begin{equation}
\label{prep_Y_p_to_q}
\| Y_i\|_{q}  \le2 \frac qp   \|Y_i\|_p \qquad \mbox{for } q\ge p\ge 2\log_2 \alpha.
\end{equation}
If $\alpha \le 2$ we have
\begin{equation}
\label{Y_p_to_q}
\| Y_i\|_{q}  \le 2 \frac qp   \|Y_i\|_p \qquad \mbox{for } q\ge p\ge 2.
\end{equation}

Otherwise, take $2\log_2\alpha\ge q\ge p\ge 2$. Then by H\"older's inequality and
\eqref{prep_Y_p_to_q} with exponents  $\frac{p(q-1)}{p-1}$ and $q$ we get
\[
\|Y_i\|_q^q=\Ex |Y_i||Y_i|^{q-1}
\le \bigl( \Ex |Y_i|^p \bigr)^{\frac 1p} \Bigl( \Ex|Y_i|^{\frac{p(q-1)}{p-1}} \Bigr)^{\frac{p-1}p} 
\le \|Y_i\|_p \|Y_i\|_q^{q-1} \biggl(2 \frac{p(q-1)}{q(p-1)} \biggr)^{q-1}. 
\]
Observe that 
\[
\biggl(2 \frac{p(q-1)}{q(p-1)} \biggr)^{q-1}\le 4^{q-1}\le \frac{1}{4}\alpha^4,
\]
so 
\[
\|Y_i\|_q\le \frac{1}{4}\alpha^4\|Y_i\|_p \qquad \mbox{for } 2\log_2\alpha\ge q\ge p\ge 2. 
\]

Thus for any value of $\alpha$ we get
\[
\|Y_i\|_{q}  \le \max\biggl\{2,\frac{1}{2}\alpha^4\biggr\} \frac qp   \|Y_i\|_p \qquad 
\mbox{for } q\ge p\ge 2.
\]
Hence, by \cite[Theorem 2.3]{LaTk}  the variables $Y_1, \ldots , Y_n$ satisfy \eqref{ass:Y_to_X}
(in fact for arbitrary, not only unconditional sets $T$) and the assertion follows by 
Proposition \ref{thm:Y_to_X}.
\end{proof}

\section{Symmetrization argument}

We will use the following  proposition  to prove that we may skip the 
uncon\-dition\-ality assumption  in  Corollary \ref{cor:p2p_to_paour}.

\begin{proposition}
\label{prop:est_bern}
Let $(X_i)_{i=1}^n$ be a sequence of independent random variables with finite second moments 
and let $(\ve_i)_{i=1}^n$ be
a Bernoulli sequence independent of $(X_i)_{i=1}^n$. Then for any $T\subset \er^n$ and  $p\ge 1$,
\begin{equation}
\label{eq:est_bern}
\Ex_X\sup_{t\in T}\biggl(\Ex_{\ve}\Bigl|\sum_{i=1}^nt_i\ve_i X_i\Bigr|^p\biggr)^{1/p}
\le
C\Biggl[\Ex\sup_{t\in T}\sum_{i=1}^nt_i\ve_i X_i
+\sup_{t\in T}\biggl(\Ex \Bigl|\sum_{i=1}^n t_i\ve_i X_i\Bigr|^p\biggr)^{1/p}\Biggr].
\end{equation}
\end{proposition}

\begin{proof}

Since this is only a matter of normalization we may and will assume that
$\Ex X_i^2=1$ for all $i$.

Let $m$ be such an integer that $2m\le p < 2(m+1)$. Then, by the symmetry of $X_i, \eps_i$, and  
the independence of $X_1, \ldots , X_n, \eps_1,\ldots  ,\eps_n$  we have
\begin{align*}
\biggl\|\sum_{i=1}^nt_i\ve_i X_i\biggr\|_{p} 
& \ge 
\biggl\|\sum_{i=1}^nt_i\ve_i  X_i\biggr\|_{2m}  
\\ 
& = \biggl( \sum_{i_1+\ldots i_n =m} 
c_{i_1,\ldots ,i_n} t_1^{2i_1}\ldots t_n^{2i_n} \Ex X_1^{2i_1} \ldots \Ex X_n^{2i_n}  \biggr)^{1/2m}
\\ 
& \ge
\biggl( \sum_{i_1+\ldots i_n =m} c_{i_1,\ldots ,i_n} t_1^{2i_1}\ldots t_n^{2i_n}  \biggr)^{1/2m} 
\\ 
& =
\biggl( \sum_{i_1+\ldots i_n =m} c_{i_1,\ldots ,i_n} t_1^{2i_1}\ldots t_n^{2i_n}   
\Ex \eps_1^{2i_1} \ldots \Ex \eps_n^{2i_n}  \biggr)^{1/2m} 
=
\biggl\|\sum_{i=1}^nt_i\ve_i \biggr\|_{2m},
\end{align*}
where
\begin{equation*}
c_{i_1, \ldots , i_n} = \frac{(2i_1+\ldots +2i_n)!}{(2i_1)!\ldots (2i_n)!}.
\end{equation*}
	
Moreover by the result of Hitczenko \cite{MR1244666},
 	\begin{equation*}
		\biggl\|\sum_{i=1}^nt_i\ve_i \biggr\|_{2m} \ge \frac 1C \Biggl[ \sum_{i\le 2m} t_i^* + \sqrt{2m} \sqrt{\sum_{i>2m} |t_i^*|^2} \Biggr],
	\end{equation*}
 where $(t_i^*)_{i=1}^n$ denotes the non-increasing rearrangement of $(|t_i|)_{i=1}^n$. 

 Therefore
to establish \eqref{eq:est_bern} it is enough to show that
\begin{equation}
\label{eq:est_bern2}
\Ex\sup_{t\in T}\biggl(\Ex_{\ve}\Bigl|\sum_{i=1}^nt_i\ve_i X_i\Bigr|^p\biggr)^{1/p}
\le 
C\biggl(\Ex\sup_{t\in T}\sum_{i=1}^n t_i\ve_i X_i +pa\biggr),
\end{equation}
where
\[
a:=\frac{1}{p}\sup_{t\in T}\biggl(\sum_{i\le p} t_i^*+\sqrt{p}\Bigl(\sum_{i>p}|t_i^*|^2\Bigr)^{1/2}\biggr). 
\]

To this end observe that since
\begin{align*}
\biggl\|\sum_{i=1}^n u_i\ve_i\biggr\|_p\le C\sqrt{p}\|u\|_2,\quad 
&\biggl\|  \sum_{i=1}^n  u_i\ve_i\biggr\|_p\le \|u\|_1,
\\  
\mbox{ and }\quad \biggl\|\sum_{i=1}^n u_i\ve_i\biggr\|_p 
& = \biggl\|\sum_{i=1}^n |u_i|\ve_i\biggr\|_p,
\end{align*}
we have
\[
\biggl\|\sum_{i=1}^n u_i\ve_i\biggr\|_p\le \sum_{i=1}^n(|u_i|-a)_+ 
+C\sqrt{p}\biggl(\sum_{i=1}^n\min\{u_i^2, a^2\}\biggr)^{1/2}.
\]
Thus
\begin{multline}
\label{eq:est_cut1}
\Ex_X\sup_{t\in T}\biggl(\Ex_{\ve}\Bigl|\sum_{i=1}^nt_i\ve_i X_i\Bigr|^p\biggr)^{1/p}
\\
\le \Ex\sup_{t\in T}\sum_{i=1}^n \bigl(|t_iX_i|-a\bigr)_+ 
+C\sqrt{p}\biggl(\Ex\sup_{t\in T}\sum_{i=1}^n\min\big\{(t_iX_i)^{2},a^2\big\}\biggr)^{1/2}.
\end{multline}

To estimate the first term above observe that
\begin{multline*}
	\Ex\sup_{t\in T}\sum_{i=1}^n \bigl(|t_iX_i|-a\bigr)_+
	\\
	\le
	\sup_{t\in T}\Ex\sum_{i=1}^n \bigl(|t_iX_i|-a\bigr)_+
	+\Ex\sup_{t\in T}\sum_{i=1}^n\Big( \bigl(|t_iX_i|-a\bigr)_+-\Ex \bigl(|t_iX_i'|-a\bigr)_+\Big),
\end{multline*}
where $(X_i')_i$ is a copy of $(X_i)$, independent of  $(\ve_i)$ and $(X_i)$. 

Observe that for any $u$ and $i$
\[
\Ex\bigl(|uX_i|-a\bigr)_+\le |u|\Ex|X_i|\le |u| \|X_i\|_2 = |u|
\]
and, by the Cauchy-Schwarz inequality and the Markov inequality
\begin{align*}
\Ex \bigl(|uX_i|-a\bigr)_+ 
&\le |u|\Ex|X_i|I_{\{|X_i|\ge a/|u|\}}
\le |u|\|X_i\|_2 \bigl(\Pr(|X_i|\ge a/|u|)\bigr)^{1/2}
\\ 
&\le |u|\|X_i\|_2^2\frac{|u|}{a}=\frac{u^2}{a}.
\end{align*}
Hence for any $t\in T$
\[
\sum_{i=1}^n\Ex \bigl(|t_iX_i|-a\bigr)_+\le \sum_{i\le p}t_i^*+\frac{1}{a}\sum_{i>p}(t_i^*)^2\le 2pa.
\]

Moreover, by the Jensen inequality
\begin{align*}
\Ex\sup_{t\in T}\sum_{i=1}^n \Bigl( \bigl(|t_iX_i|-a\bigr)_+&-\Ex \bigl(|t_iX_i'|-a\bigr)_+ \Bigr)
\\
&\le \Ex\sup_{t\in T}\sum_{i=1}^n\Bigl( \bigl(|t_iX_i|-a \bigr)_+ -\bigl(|t_iX_i'|-a\bigr)_+\Big)
\\
&= \Ex\sup_{t\in T}\sum_{i=1}^n\ve_i\Bigl( \bigl(|t_iX_i|-a\bigr)_+-\bigl(|t_iX_i'|-a\bigr)_+\Big)
\\
&\le \Ex\sup_{t\in T}\sum_{i=1}^n\ve_i \bigl(|t_iX_i|-a\bigr)_+ + \Ex\sup_{t\in T}\sum_{i=1}^n-\ve_i \bigl(|t_iX_i'|-a\bigr)_+
\\
&=2\Ex \sup_{t\in T}\sum_{i=1}^n\ve_i \bigl(|t_iX_i|-a\bigr)_+.
\end{align*}
Function $x\mapsto (|x|-a)_+$ is $1$-Lipschitz, so Talagrand's comparison theorem for Bernoulli processes
\cite[Theorem 2.1]{MR1207231} yields
\[
\Ex\sup_{t\in T}\sum_{i=1}^n\ve_i \bigl(|t_iX_i|-a\bigr)_+\le \Ex\sup_{t\in T}\sum_{i=1}^nt_i\ve_i X_i.
\]
Therefore
\begin{equation}
\label{eq:est_cut2}
\Ex\sup_{t\in T}\sum_{i=1}^n \bigl(|t_iX_i|-a\bigr)_+\le 2pa+2\Ex\sup_{t\in T}\sum_{i=1}^nt_i\ve_i X_i.
\end{equation}

Now we turn our attention to the other term in \eqref{eq:est_cut1}. We have
\begin{align*}
&\Ex\sup_{t\in T}\sum_{i=1}^n\min\bigl\{(t_iX_i)^{2}, a^2\bigr\}
\\
&\le \sup_{t\in T}\Ex\sum_{i=1}^n\min \bigl\{(t_iX_i)^{2}, a^2\bigr\}
+ \Ex\sup_{t\in T}\sum_{i=1}^n\Bigl(\min\bigl\{(t_iX_i)^{2}, a^2\bigr\}-\Ex\min\bigl\{(t_iX_i)^{2}, a^2\bigr\}\Bigr).
\end{align*}

We have
\[
\sum_{i=1}^n \Ex \min\bigl\{(t_iX_i)^{2}, a^2\bigr\}
\le \sum_{i=1}^n\min\bigl\{a^2,t_i^2\Ex X_i^2\bigr\}\le pa^2+\sum_{i>p}(t_i^*)^2
\le 2pa^2.
\]
Moreover, by the Jensen inequality
\begin{align*}
\Ex\sup_{t\in T}\sum_{i=1}^n\Bigl(\min\bigl\{(t_iX_i)^{2}, a^2\bigr\}-&\Ex\min\bigl\{(t_iX_i')^{2}, a^2\bigr\}\Bigr)
\\
&\le
\Ex\sup_{t\in T}\sum_{i=1}^n\Bigl(\min\bigl\{(t_iX_i)^{2}, a^2\bigr\}-\min\bigl\{(t_iX_i')^{2}, a^2\bigr\}\Bigr)
\\
& =
\Ex\sup_{t\in T}\sum_{i=1}^n\ve_i\Bigl(\min\bigl\{(t_iX_i)^{2}, a^2\bigr\}-\min\bigl\{(t_iX_i')^{2}, a^2\bigr\}\Bigr)
\\
&
\le 2\Ex \sup_{t\in T}\sum_{i=1}^n\ve_i \min\bigl\{(t_iX_i)^{2}, a^2\bigr\}.
\end{align*}
Function $x\mapsto \min\{x^2,a^2\}$ is $2a$-Lipschitz, so using  the comparison theorem for Ber\-noulli
processes again we get
\[
\Ex\sup_{t\in T}\sum_{i=1}^n\ve_i \min\bigl\{(t_iX_i)^{2}, a^2\bigr\}
\le 2a\Ex\sup_{t\in T}\sum_{i=1}^nt_i\ve_i X_i.
\]
Thus
\begin{multline}
p\Ex\sup_{t\in T}\sum_{i=1}^n\min\bigl\{(t_iX_i)^{2}, a^2\bigr\}
\le 
2p^2a^2+4pa\Ex\sup_{t\in T}\sum_{i=1}^n t_i\ve_i X_i
\\
\label{eq:est_cut3}
\le \biggl(2pa+\Ex\sup_{t\in T}\sum_{i=1}^n t_i\ve_i X_i\biggr)^2.
\end{multline}

Estimate \eqref{eq:est_bern2} follows by \eqref{eq:est_cut1}-\eqref{eq:est_cut3}.

\end{proof}

\begin{proof}[Proof of Theorem  \ref{thm:p2p_to_paour}]
Since  it is enough to consider $T\cup (-T)$ instead of $T$, we may and will assume that the set $T$ is symmetric, i.e. $T=-T$.
 
Assume first that  the variables $X_i$ are also symmetric.
Let $\ve=(\ve_i)_{i=1}^n$ be a Bernoulli sequence independent of $(X_i)_{i=1}^n$.
Weak and strong moments of $(\ve_i)_{i=1}^n$ are comparable
\[
\biggl(\Ex\sup_{s\in S}  \Bigl|\sum_{i=1}^n s_i\eps_i \Bigr|^p \biggr)^{1/p} 
\le 
C \Biggl[  \Ex  \sup_{s\in S} \Bigl|\sum_{i=1}^n s_i\eps_i\Bigr|  
+  \sup_{s\in S} \biggl(\Ex   \Bigl|\sum_{i=1}^n s_i \eps_i \Bigr|^p \biggr)^{1/p} \Biggr].  
\] 
Hence the symmetry of $X_i$ yields
\begin{align}
\notag
\biggl(\Ex  \sup_{t\in T} &  \Bigl|  \sum_{i=1}^n  t_iX_i  \Bigr|^p  \biggr)^{1/p}
=\biggl(\Ex_X \Ex_\eps  \sup_{t\in T} \Bigl|\sum_{i=1}^n  t_i X_i\eps_i \Bigr|^p \biggr)^{1/p}
\\
\label{eq:bernsymm}
&\le 
2C  \Biggl[  
\biggl(\Ex_X   \Bigl(\Ex_\eps  \sup_{t\in T} \Bigl|\sum_{i=1}^n t_iX_i\eps_i \Bigr| \Bigr)^p \biggr)^{1/p}  
+  \biggl(\Ex_X  \sup_{t\in T} \Ex_\eps   \Bigl|\sum_{i=1}^n t_i X_i \eps_i \Bigr|^p \biggr)^{1/p}\Biggr],
\end{align}
since $(a+b)^p \le 2^p (a^p +b^p)$.

Since $T$ is symmetric, we have for $x\in \er^n$,
\[
\Ex_\eps  \sup_{t\in T} \Bigl|\sum_{i=1}^n t_ix_i\eps_i \Bigr|=
\sup_{t\in T_1}\sum_{i=1}^n t_ix_i, 
\]
where 
\[
T_1:=\left\{(\Ex_{\ve} s_i(\ve)\ve_i)_{i=1}^n \colon\ s\colon \{-1,1\}^n\to T\right\}
\]
is an unconditional subset of $\er^n$. Estimate \eqref{eq:weakstronguncond} applied for $T_1$ instead of $T$ 
yields
\begin{multline*}
\biggl(\Ex_X   \Bigl(\Ex_\eps  \sup_{t\in T} \Bigl|\sum_{i=1}^n t_iX_i\eps_i \Bigr| \Bigr)^p \biggr)^{1/p}
\\
\le
C(\alpha)\Biggl[
\Ex_X \Ex_\eps  \sup_{t\in T} \Bigl|\sum_{i=1}^n t_iX_i\eps_i \Bigr|
+\sup_{t\in T_1}\biggl(\Ex\Bigl|\sum_{i=1}t_iX_i\Bigr|^p\biggr)^{1/p}
\Biggr].
\end{multline*}
By the symmetry of $X_i$ we have
\[
\Ex_X \Ex_\eps  \sup_{t\in T} \Bigl|\sum_{i=1}^n t_iX_i\eps_i \Bigr|
=\Ex \sup_{t\in T} \sum_{i=1}^n t_iX_i.
\]
Moreover,
\[
T_1\subset S(T):=\mathrm{conv}\left\{(\eta_i t_i)_{i=1}^n: \eta\in\{ -1,1\}^n, t\in T\right\},
\]
hence
\[
\sup_{t\in T_1}\biggl(\Ex\Bigl|\sum_{i=1}^nt_iX_i\Bigr|^p\biggr)^{1/p}
\le \sup_{t\in S(T)}\biggl(\Ex\Bigl|\sum_{i=1}t_iX_i\Bigr|^p\biggr)^{1/p}
=\sup_{t\in T}\biggl(\Ex\Bigl|\sum_{i=1}^nt_iX_i\Bigr|^p\biggr)^{1/p}.
\]
Thus
\begin{multline}
\biggl(\Ex_X   \Bigl(\Ex_\eps  \sup_{t\in T} \Bigl|\sum_{i=1}^n t_iX_i\eps_i \Bigr| \Bigr)^p \biggr)^{1/p}
\\
\label{eq:term1}
\le
C(\alpha)\Biggl[\Ex \sup_{t\in T} \Bigl|\sum_{i=1}^n t_iX_i \Bigr|+
\sup_{t\in T}\biggl(\Ex\Bigl|\sum_{i=1}t_iX_i\Bigr|^p\biggr)^{1/p}
\Biggr].
\end{multline}

Let $q=p/(p-1)$ be the H\"older's dual of $p$. For $x\in \er^n$ we have
\[
\biggl(\sup_{t\in T} \Ex_\eps   \Bigl|\sum_{i=1}^n t_i x_i \eps_i \Bigr|^p \biggr)^{1/p}
=\sup_{t\in T_2}\sum_{i=1}^n t_ix_i,
\]
where
\[
T_2=\left\{\Ex_\ve t h(\ve)\colon\ t\in T, h\colon\{-1,1\}^n\to \er, \Ex_\ve|h(\ve)|^q\le 1\right\}
\]
is  a unconditional subset of $\er^n$. Estimate \eqref{eq:weakstronguncond} applied for $T_2$ instead of $T$ 
yields
\begin{multline*}
\biggl(\Ex_X  \sup_{t\in T} \Ex_\eps   \Bigl|\sum_{i=1}^n t_i X_i \eps_i \Bigr|^p \biggr)^{1/p}
\\
\le 
C(\alpha)\Biggl[
\Ex_X  \biggl(\sup_{t\in T} \Ex_\eps   \Bigl|\sum_{i=1}^n t_i X_i \eps_i \Bigr|^p \biggr)^{1/p}
+
\sup_{t\in T_2}\biggl(\Ex \Bigl|\sum_{i=1}^n t_i X_i\Bigr|^p\biggr)^{1/p}
\Biggr].
\end{multline*}
Proposition \ref{prop:est_bern} and the symmetry of $X_i$ gives
\[
\Ex_X  \biggl(\sup_{t\in T} \Ex_\eps   \Bigl|\sum_{i=1}^n t_i X_i \ve_i\Bigr|^p \biggr)^{1/p}
\le
C\Biggl[\Ex\sup_{t\in T}\sum_{i=1}^n t_iX_i
+\sup_{t\in T}\biggl(\Ex \Bigl|\sum_{i=1}^n t_i X_i\Bigr|^p\biggr)^{1/p}\Biggr].
\]
Since $T_2\subset \operatorname{conv}T$ (recall that we assume the symmetry of $T$) we have
\[
\sup_{t\in T_2}\biggl(\Ex \Bigl|\sum_{i=1}^n t_i X_i\Bigr|^p\biggr)^{1/p}
\le \sup_{t\in \operatorname{conv}T}\biggl(\Ex \Bigl|\sum_{i=1}^n t_i X_i\Bigr|^p\biggr)^{1/p}
=\sup_{t\in T}\biggl(\Ex \Bigl|\sum_{i=1}^n t_i X_i\Bigr|^p\biggr)^{1/p}.
\]
Thus 
\begin{multline}
\biggl(\Ex_X  \sup_{t\in T} \Ex_\eps   \Bigl|\sum_{i=1}^n t_i X_i \eps_i \Bigr|^p \biggr)^{1/p}
\\
\label{eq:term2}
\le 
C(\alpha)\biggl[
\Ex\sup_{t\in T}\sum_{i=1}^n t_iX_i
+\sup_{t\in T}\biggl(\Ex \Bigl|\sum_{i=1}^n t_i X_i\Bigr|^p\biggr)^{1/p}
\Biggr].
\end{multline}

Estimate \eqref{eq:p2pweakstrong} follows (for symmetric $X_i$'s) by
\eqref{eq:bernsymm}-\eqref{eq:term2}

In the case when  the variables $X_i$ are centered, but not necessarily symmetric let
$(X_1' ,\ldots , X_n')$ be an independent copy of $(X_1, \ldots ,X_n)$.
Then $X_i-X_i'$ are symmetric. The Jensen inequality and the assumption on $X_i$ imply that 
for any $p\ge 2$ we have
\begin{equation*}
\|X_i-X_i'\|_{2p} \le 2\|X_i\|_{2p}\le 2\alpha \|X_i -\Ex X_i\|_p \le 2\alpha \|X_i-X_i'\|_p.
\end{equation*}
Therefore, Theorem \ref{thm:p2p_to_paour} applied to $(X_1-X_1' , \ldots ,X_n-X_n')$ implies
\begin{align*}
\biggl(\Ex  \sup_{t\in T} \Bigl|\sum_{i=1}^n t_i & X_i \Bigr|^p  \biggr)^{1/p} 
\\
& 
= \biggl(\Ex  \sup_{t\in T} \Bigl|\sum_{i=1}^n t_i(X_i - \Ex X_i') \Bigr|^p \biggr)^{1/p}
\le \biggl(\Ex  \sup_{t\in T} \Bigl|\sum_{i=1}^n t_i(X_i-X_i') \Bigr|^p \biggr)^{1/p}
\\ 
&
\le C(2\alpha) \Biggl[  
\Ex  \sup_{t\in T} \Bigl|\sum_{i=1}^n t_i(X_i-X_i')\Bigr|  
+  \sup_{t\in T} \biggl(\Ex   \Bigl|\sum_{i=1}^n t_i (X_i-X_i') \Bigr|^p \biggr)^{1/p} \Biggr]
\\ 
&
\le 2C(2\alpha) \Biggl[ 
\Ex  \sup_{t\in T} \Bigl| \sum_{i=1}^n t_iX_i  \Bigr|
+  \sup_{t\in T} \biggl(\Ex   \Bigl|\sum_{i=1}^n t_i X_i \Bigr|^p \biggr)^{1/p} \Biggr],
\end{align*}
 what finishes the proof in the general case.
\end{proof}

\begin{remark}
It follows by the proof of \cite[Theorem 2.3]{LaTk} that if $(X_i)_{i=1}^n$ are  symmetric,
independent and for any $i$  moments of $X_i$ grow $\beta $-regularly 
(i.e.  \eqref{Y_p_to_q} holds with $\beta$ instead of $2$),
then the comparison of weak and strong moments of suprema  of linear combinations of variables $X_i$
holds with a  constant $C(\beta)= C\beta^{11}$.  
Therefore, we may follow the constants in the proofs above to obtain that 
Theorem \ref{thm:p2p_to_paour} holds with $C(\alpha)=C^{\log_2^2\alpha}$.
\end{remark}

\section{From comparison of weak and strong moments \\ to comparison of weak and strong tails}
\label{sec:wst}

In this Section we prove Corollary \ref{cor:weakstrongtail} and Corollary \ref{cor:KahKh}. To this end we need the following lemma.

\begin{lemma}
 
Assume $X_1, X_2, \ldots$ satisfy the assumptions of Theorem \ref{thm:p2p_to_paour}. 
Then for any $t\in \er^n$,
\begin{equation}
\label{eq:p2qlincomb}
\biggl\|\sum_{i=1}^n t_iX_i \biggr\|_p\le C(\alpha) \biggl(\frac{p}{q}\biggr)^{\max\{ 1/2,\log_2\alpha\}}
\biggl\|\sum_{i=1}^n t_iX_i \biggr\|_q \quad \mbox{ for }p\ge q\ge 2.
\end{equation}
\end{lemma}

\begin{proof}
Let $\beta:=\max\{ 1/2,\log_2\alpha\}$.
It is enough to show that for positive integers $k\le l$ we have
\[
\biggl\|\sum_{i=1}^n t_iX_i \biggr\|_{2k}
\le 
C\alpha \biggl(\frac{k}{l}\biggr)^{\beta}
\biggl\|\sum_{i=1}^n t_iX_i \biggr\|_{2l}. 
\]
 A standard symmetrization argument
shows that we may assume that  the  random variables $X_i$ are symmetric (see the proof of Theorem \ref{thm:p2p_to_paour} in the non-symmetric case). 

Using  the hypercontractivity method \cite[ Section 3.3]{KwWo}, it is enough to show that for $1\le i\le n$,
\[
\biggl\|s+\frac{t}{2\sqrt{2}e\alpha}\biggl(\frac{l}{k}\biggr)^{\beta}X_i\biggr\|_{2k}
\le \bigl\|s+tX_i\bigr\|_{2l}
\quad \mbox{for all } s,t\in \er.
\]
This reduces to the following claim.

\medskip

\noindent
{\bf Claim.} Suppose that $Y$ is a symmetric  random variable such that $\|Y\|_{2p}\le \alpha\|Y\|_p$ for
some $\alpha\ge 1$ and every $p\ge 2$. Let $k\ge l$ be positive
integers. Then
\[
\bigl\|1+\sigma Y\bigr\|_{2k}\le \bigl\|1+Y \bigr\|_{2l},
\quad \mbox{ where } \sigma:=\frac{1}{2\sqrt{2}e\alpha}\biggl(\frac{l}{k}\biggr)^{\beta}.
\]
\medskip

To show the claim observe first that
\begin{equation}
\label{eq:compmom_q_p}
\|Y\|_{q}\le \alpha\biggl(\frac{q}{p}\biggr)^{\log_2 \alpha}\|Y\|_p
\le \alpha\biggl(\frac{q}{p}\biggr)^{\beta}\|Y\|_p
\quad \mbox{ for }q\ge p\ge 2.
\end{equation}

 Moreover we have
\begin{align*}
\Ex\bigl|1+\sigma Y\bigr|^{2k}
&=1+\sum_{j=1}^{k}\binom{2k}{2j}\Ex\left|\sigma Y\right|^{2j}
\le 1+\sum_{j=1}^{k}\biggl(\frac{ek}{j}\sigma\|Y\|_{2j}\biggr)^{2j}
\\
&\le 1+\sup_{1\le j\le k}\biggl(\frac{\sqrt{2}ek}{j}\sigma\|Y\|_{2j}\biggr)^{2j},
\end{align*}
 so it is enough to show that
\begin{equation}
\label{eq:aim}
1+\biggl(\frac{k^{1-\beta}l^{\beta}}{2j\alpha}\|Y\|_{2j}\biggr)^{2j}\le \bigl\|1+Y\bigr\|_{2l}^{2k}
\quad \mbox{ for }j=1,2\ldots l.
\end{equation}

To this end we will use the following deterministic inequality:
\begin{equation}
\label{eq:simple}
(1+u)^p\ge \biggl(1+\frac{p}{q}u\biggr)^q\ge 1+\biggl(\frac{p}{q}u\biggr)^q\quad
\mbox{ for $p\ge q\ge 1$ and $u\ge 0$,}
\end{equation}
and  a simple lower bound for $\|1+Y\|_{2l}^{2l}$:
\begin{equation}
\label{eq:mom2lbelow}
\Ex|1+Y|^{2l}=1+\sum_{r=1}^{l}\binom{2l}{2r}\Ex|Y|^{2r}
\ge 1+\sum_{r=1}^l\biggl(\frac{l}{r}\|Y\|_{2r}\biggr)^{2r}.
\end{equation}

Assume first that $1\le j\le \frac{k}{l}$. Estimate \eqref{eq:compmom_q_p} applied with $p=2j$ and $q=2$ yields 
\[
\frac{k^{1-\beta}l^{\beta}}{2j\alpha}\|Y\|_{2j}\le \frac{k^{1-\beta}l^{\beta}}{j^{1-\beta}}\|Y\|_{2}
\le \sqrt{\frac{kl}{j}}\|Y\|_{2},
\]
where the last inequality holds since $\beta\ge \frac{1}{2}$ and $k\ge jl$. 
Inequalities \eqref{eq:mom2lbelow} and \eqref{eq:simple} (applied with $p=k/l$ and $q=j$) yield
\[
\bigl\|1+Y\bigr\|_{2l}^{2k}\ge \bigl(1+(l\|Y\|_2)^2\bigr)^{k/l}\ge 1+\Biggl(\sqrt{\frac{kl}{j}}\|Y\|_{2}\Biggr)^{2j}
\]
 so \eqref{eq:aim} holds for $j\le \frac{k}{l}$.

If $j\ge \frac{k}{l}$ we choose $r=\lceil jl/k\rceil $, then $jl\le kr\le 2jl$. 
Since $1\le r\le l$,   the estimate \eqref{eq:mom2lbelow} gives 
\[
\bigl\|1+Y\bigr\|_{2l}^{2k}\ge \Biggl(1+\biggl(\frac{l}{r}\|Y\|_{2r}\biggr)^{2r}\Biggr)^{k/l}
\ge \Biggl(1+\biggl(\frac{l}{r}\|Y\|_{2r}\biggr)^{2r}\Biggr)^{j/r}
\ge 1+\biggl(\frac{l}{r}\|Y\|_{2r}\biggr)^{2j},
\]
where to get the last two inequalities we used $k/l\ge j/r$ and $j/r\ge 1$.
Applying estimate \eqref{eq:compmom_q_p} with $2j$ and $2r$ instead of $p$ and $q$ we get
\[
\frac{k^{1-\beta}l^{\beta}}{2j\alpha}\|Y\|_{2j}
\le \frac{k^{1-\beta}l^{\beta}}{2j}\biggl(\frac{j}{r}\biggr)^\beta\|Y\|_{2r}
\le \frac{k}{2j}\|Y\|_{2r}\le \frac{l}{r}\|Y\|_{2r},
\] 
which completes the proof of the claim in the remaining case. 

\end{proof}

\begin{proof}[Proof of Corollary \ref{cor:weakstrongtail}]
Let  
\[
S:=\sup_{t\in T} \Bigl|\sum_{i=1}^n t_iX_i \Bigr|.
\]

By the Paley-Zygmund inequality and \eqref{eq:p2qlincomb} we have for $t\in T$,
\begin{align}
\notag
\Pr\Biggl(\Bigl|\sum_{i=1}^n t_iX_i \Bigr|\ge \frac{1}{2}\biggl\|\sum_{i=1}^n t_iX_i \biggr\|_p\Biggr)
&=\Pr\biggl(\Bigl|\sum_{i=1}^n t_iX_i \Bigr|^p\ge 2^{-p}\Ex\Bigl|\sum_{i=1}^n t_iX_i \Bigr|^p\biggr) 
\\
\label{eq:PZbelow}
&\ge 
(1-2^{-p})^2\Biggl(\frac{\bigl\|\sum_{i=1}^n t_iX_i\bigr\|_p}{\bigl\|\sum_{i=1}^n t_iX_i\bigr\|_{2p}}\Biggr)^{2p}
\ge e^{-C_4(\alpha)p}. 
\end{align}

In order to show \eqref{eq:weakstrongtail} we consider 3 cases.\\

\textbf{Case 1.} $2u<\sup_{t\in T}\|\sum_{i=1}^n t_iX_i\|_{2}$. Then by \eqref{eq:PZbelow}
\[
\sup_{t\in T}\Pr\biggl(\Bigl|\sum_{i=1}^nt_iX_i\Bigr|\ge u\biggr)
\ge e^{-2C_4(\alpha)}
\]
and \eqref{eq:weakstrongtail} obviously holds if $C_2(\alpha)\ge \exp(2C_4(\alpha))$. 

\textbf{Case 2.} $\sup_{t\in T}\|\sum_{i=1}^n t_iX_i\|_{2}\le 
2u<\sup_{t\in T}\|\sum_{i=1}^n t_iX_i\|_{\infty}$.
Let us then define
\[
p:=\sup\biggl\{q\ge 2C_4(\alpha)\colon\ 
\sup_{t\in T}\Bigl\|\sum_{i=1}^n t_iX_i\Bigr\|_{q/C_4(\alpha)}\le 2u\biggr\}.
\]
By \eqref{eq:PZbelow} we have

\[
\sup_{t\in T}\Pr\biggl(\Bigl|\sum_{i=1}^n t_iX_i\Bigr|\ge u\biggr)\ge e^{-p}.
\]

By \eqref{eq:p2qlincomb} we have 
$\sup_{t\in T}\|\sum_{i=1}^n t_iX_i\|_{p}\le C(\alpha)u$, so  by Theorem \ref{thm:p2p_to_paour} and Chebyshev's inequality
we have
\[
\Pr(S\ge C_1(\alpha)(\Ex S+u))\le
\Pr(S\ge e\|S\|_p)\le e^{-p}
\]
 for $C_1(\alpha)$
large enough.
Thus \eqref{eq:weakstrongtail} holds in this case.

\textbf{Case 3.} $u>\sup_{t\in T}\|\sum_{i=1}^n t_iX_i\|_{\infty}=\|S\|_\infty$. Then
$\Pr(S\ge u)=0$ and \eqref{eq:weakstrongtail} holds for any $C_1(\alpha)\ge 1$.
\end{proof}

\begin{proof}[Proof of Corollary \ref{cor:KahKh}]
The result is an immediate consequence of Theorem \ref{thm:p2p_to_paour}, \eqref{eq:p2qlincomb} and \eqref{eq:lowtrivial} used with $q$ instead of $p$.
\end{proof}

\section{ Comparison  of  weak and strong  moments of suprema \\  implies comparison of  moments $p$ and $2p$  }

\begin{proof}[Proof of Theorem  \ref{thm:paour_to_p2p}]
We will use the assumption \eqref{thm_assump:paour_to_p2p} for $T$ containing all vectors of the standard base of 
$\R^n$and their negatives, i.e. we will use only the inequality
\begin{equation}
\label{assumpt}
\Bigl(\Ex  \supi |X_i|^p \Bigr)^{1/p} 
\le L \Bigl[  \Ex  	\supi |X_i|  + \|X_1\|_p \Bigr] .
\end{equation}

Fix $p\ge 2$ and let $n:=\lfloor (4L)^{2p} \rfloor +1 $,  $A:= n^{1/p} \|X_1\|_p$. 
If $A\ge \|X_1\|_{2p}$, then \eqref{eq:comp_p2p} holds with $\alpha = (4L)^2+1$. 
Hence we may and will assume $A\le \|X_1\|_{2p}$.

Obviously 
\[
\Pr\Bigl(\supi |X_i|\ge t\Bigr)  \le \min \bigl\{1, n\Pr\bigl(|X_1|\ge t\bigr)\bigr\}.
\]
Moreover, if $\Pr(|X_1|\ge t) \le \frac{1}n$, 
\begin{align*}
\Pr\Bigl(\supi |X_i|\ge t\Bigr) 
&= 1-\Pr\bigl(|X_1|< t\bigr)^n
= \Pr\bigl(|X_1|\ge t\bigr)\sum_{k=0}^{n-1}\Pr\bigl(|X_1|<t\bigr)^k 
\\
&\ge 
\Pr\bigl(|X_1|\ge t\bigr) \cdot n\biggl(1-\frac 1n \biggr)^{n-1} \ge \frac n3 \Pr\bigl(|X_1|\ge t\bigr).
\end{align*}
Since $\Pr(|X_1|\ge A)\le \frac 1n$  (which follows by the Markov inequality) and $A\le \|X_1\|_{2p}$, 
we have
\begin{align*}
\Ex \supi |X_i|^{2p} 
&\ge 
2p \int_A^\infty t^{2p-1}\Pr\Bigl(\supi |X_i|\ge t\Bigr)dt 
\ge  2p\int_A^\infty t^{2p-1}\frac n3 \Pr\bigl(|X_1|\ge t\bigr)dt
\\ 
&= 
\frac n3 \Ex\bigl( |X_1|^{2p} - A^{2p} \bigr)_+ 
\ge \frac n3 \bigl(\|X_1\|_{2p}^{2p} - A^{2p} \bigr) 
\ge \frac n3 \bigl(\|X_1\|_{2p} - A \bigr)^{2p}
\end{align*}
and
\begin{align*}
\Ex \supi|X_i| 
&\le 
A + \int_A^\infty \Pr\Bigl(\supi |X_i|\ge t\Bigr)dt 
\le A+ n\int_A^\infty \Pr\bigl(|X_1|\ge t\bigr)dt 
\\
&\le A +n\Ex\bigl(|X_1|\ind_{\{|X_1|\ge A\}}\bigr) 
\le A+ n\|X_1\|_p \Pr\bigl(|X_1|\ge A\bigr)^{1-\frac 1p} 
\\
&\le A + n^{1/p}\|X_1\|_p,
\end{align*}
where in the last inequality we used again the fact that $\Pr(|X_1|\ge A)\le \frac 1n$. 

Thus our choice of $n$ and $A$, and \eqref{assumpt} (applied to $2p$ instead of $p$) imply that
\begin{align*}
2L\|X_1\|_{2p} 
& \le 
\frac 12 n^{\frac 1{2p}} \|X_1\|_{2p}
\le \frac 12 n^{\frac 1{2p}}A+\Bigl(\Ex  \supi |X_i|^{2p} \Bigr)^{1/(2p)} 
\\
&\le \frac 12 n^{\frac 1{2p}}A+L \Bigl[  \Ex  \supi |X_i|  + \|X_1\|_{2p} \Bigr]
\\
&\le 
\frac{1}2n^{\frac{1}{2p}} A + LA + Ln^{\frac 1p}\|X_1\|_p + L\|X_1\|_{2p}  
\\ 
&\le \|X_1\|_p\Bigl(\frac 12 (4L+1)n^{\frac 1p} + 2Ln^{\frac1p}  \Bigr) +L\|X_1\|_{2p} 
\\
&\le  
\biggl(4L+\frac{1}{2}\biggr)\bigl((4L)^2 +1 \bigr)\|X_1\|_p +L\|X_1\|_{2p}.  
\end{align*}
Thus 
\begin{equation*}
\|X_1\|_{2p} \le  \biggl(4+\frac{1}{2L}\biggr)\bigl(16L^2 +1 \bigr) \|X_1\|_p.
\end{equation*}
\end{proof}

\begin{remark}
It is clear from the proof above that we may take $\alpha(L)=CL^2$ in Theorem \ref{thm:paour_to_p2p}.
\end{remark}

\bibliographystyle{amsplain}
\bibliography{weakstrong_indcoord_arxiv}

\end{document}